\documentclass[11pt]{article}

\usepackage[a4paper,margin=31mm]{geometry}
\usepackage{amsmath,amssymb,amsthm,mathtools}
\usepackage[T1]{fontenc}
\usepackage{lmodern}
\usepackage{microtype}
\usepackage{enumitem}
\usepackage[colorlinks=true,linkcolor=blue,citecolor=blue,urlcolor=blue]{hyperref}

\newtheorem{theorem}{Theorem}[section]

\newtheorem{lemma}[theorem]{Lemma}
\newtheorem{corollary}[theorem]{Corollary}
\newtheorem{conj}[theorem]{Conjecture}
\theoremstyle{definition}

\theoremstyle{remark}

\newcommand{\R}{\mathbb R}
\newcommand{\Hh}{\mathcal H}

\newcommand{\ip}[2]{\langle #1,#2\rangle}
\newcommand{\norm}[1]{\lVert #1\rVert_2}

\title{Sharp projection inequalities for antichains in cone orders}
\author{
Yuxian Dong\\
\small School of Mathematical Sciences, Dalian University of Technology\\
\small Dalian 116024, China
\and
Lili Mu\thanks{Corresponding author. Email: \texttt{lilimu@jsnu.edu.cn}.}\\
\small School of Mathematics and Statistics, Jiangsu Normal University\\
\small Xuzhou 221116, China
}
\date{}

\begin{document}
\maketitle

\begin{abstract}
We establish sharp projection inequalities for $k$-antichains under orders defined by open convex cones.
For integers $n\geq2$ and $k\ge 1$, let $C\subsetneq\mathbb R^n$ be a nonempty open convex cone.
A set $A\subseteq\mathbb R^n$ is a $k$-antichain with respect to the order $x\prec_C y\iff y-x\in C$ if it contains no chain $x^1\prec_C\cdots\prec_C x^{k+1}$.
For every such set $A$ and every $v\in C$, we prove
\[
\mathcal H^{n-1}(A)
\leq k\frac{\|v\|}{\alpha_C(v)}
\mathcal H^{n-1}(P_v(A)),
\]
where $P_v$ is the orthogonal projection onto $v^\perp$, $\mathcal H^{n-1}$ denotes normalized Hausdorff outer measure, and $\alpha_C(v)$ is the minimum of $\langle\xi,v\rangle$ over unit vectors in the dual cone $C^*$.
In particular, every weak antichain $A\subseteq\mathbb R^n$ satisfies $\mathcal H^{n-1}(A)\leq\mathcal H^{n-1}(\Delta_i(A))$ for each $i$, where $\Delta_i(x)=(x_j-x_i)_{j\ne i}$.
This proves the skewed projection conjecture of Engel, Mitsis, Pelekis and Reiher.
As an application, we recover the bound $\mathcal H^{n-1}(A)\leq kn$ for $k$-antichains in the unit cube under the coordinatewise order.
\end{abstract}



\section{Introduction}

Let $n\geq2$ be an integer and set $[n]=\{1,\ldots,n\}$.
For $x,y\in\R^n$, write $x\ll y$ if $x_i<y_i$ for every $i\in[n]$.
A set $A\subseteq\R^n$ is called a \emph{weak antichain} if there are no $x,y\in A$ such that $x\ll y$. 

We measure the size of weak antichains using normalized Hausdorff outer measure.
For an integer $s\geq1$, let $\omega_s$ denote the volume of the unit ball in $\R^s$. 
If $E$ is a subset of a Euclidean space and $\delta>0$, define
\[
\Hh^s_\delta(E)
=
\inf\left\{
\sum_{r=1}^{\infty}
\omega_s\left(\frac{\operatorname{diam}U_r}{2}\right)^s
:
E\subseteq\bigcup_{r=1}^{\infty}U_r,\quad
\operatorname{diam}U_r\leq\delta \text{ for all } r
\right\},
\]
where the infimum is taken over all countable coverings $\{U_r\}_{r\ge1}$ of $E$, and $\operatorname{diam}U_r$ denotes the Euclidean diameter of $U_r$.
Set $\Hh^s(E)=\lim_{\delta\rightarrow0}\Hh^s_\delta(E)$.
This is the normalized $s$-dimensional Hausdorff outer measure~\cite[Definition~2.1]{EvansGariepy}.
In $\R^s$, $\Hh^s$ agrees with $s$-dimensional Lebesgue outer measure~\cite[Theorem~2.5]{EvansGariepy}.
A subset $B$ of a Euclidean space is called $\Hh^s$-measurable if
\[
\Hh^s(E)
=
\Hh^s(E\cap B)+\Hh^s(E\setminus B)
\]
for every subset $E$ of that space~\cite[Definition~1.3]{EvansGariepy}.

For $i\in[n]$, let $\pi_i:\R^n\to\R^{n-1}$ be the coordinate projection $\pi_i(x)=(x_j)_{j\ne i}$.
Engel, Mitsis, Pelekis and Reiher~\cite[Theorem~1.3]{EMPR} proved that every weak antichain $A\subseteq[0,1]^n$ satisfies
\begin{equation}\label{eq:EMPR-orthogonal}
\Hh^{n-1}(A)
\leq 
\sum_{i=1}^n \Hh^{n-1}(\pi_i(A)).
\end{equation}
Since each $\pi_i(A) \subseteq [0,1]^{n-1}$, we have $\Hh^{n-1}(\pi_i(A))\le1$ for every $i$.
This implies $\Hh^{n-1}(A)\leq n$. 
This bound is attained by the weak antichain
\[
F=\bigcup_{i=1}^{n}\{x\in[0,1]^n:x_i=0\}
\]
whose $(n-1)$-dimensional Hausdorff measure is $n$.

Engel, Mitsis, Pelekis and Reiher also considered the \emph{skewed projections} $\Delta_i:\R^n\to\R^{n-1}$ defined by
\[
\Delta_i(x)=(x_j-x_i)_{j\neq i}.
\]
For $A\subseteq[0,1]^n$, set
\begin{equation}\label{eq:pieces}
\underline A_i=\{x\in A:x_i=\min\{x_1,\dots,x_n\}\}.
\end{equation}
Then
\[
A=\bigcup_{i=1}^n\underline A_i,
\qquad
\Delta_i(\underline A_i)\subseteq[0,1]^{n-1}
\quad\text{for every }i\in[n].
\]

\begin{conj}[{\cite[Conjecture~5.3]{EMPR}}]
\label{conj:skewed-projection}
Every weak antichain $A\subseteq[0,1]^n$ satisfies
\begin{equation}\label{eq:original-conjecture}
\Hh^{n-1}(A)
\leq
\sum_{i=1}^n
\Hh^{n-1}\bigl(\Delta_i(\underline A_i)\bigr).
\end{equation}
\end{conj}
Engel et al.~\cite[Theorem~5.4]{EMPR} proved
Conjecture~\ref{conj:skewed-projection} in the case $n=2$.
In this paper, we prove~\eqref{eq:original-conjecture} for every
$n\geq2$ as a consequence of a more general projection inequality
for $k$-antichains under orders defined by open convex cones.

A set $S\subseteq\R^n$ is \emph{convex} if $(1-t)x+ty\in S$ for all $x,y\in S$ and $0\leq t\leq1$.
A convex set $C$ is a \emph{convex cone} if $tC\subseteq C$ for every $t>0$.
Let $C\subsetneq\R^n$ be a nonempty open convex cone.
Define
\(x\prec_C y\) if
\(y-x\in C.\)
This is a strict partial order because $0\notin C$ and $C+C\subseteq C$.
A chain is a subset whose elements are pairwise comparable.
For a positive integer $k$, a subset $A\subseteq\R^n$ is called a \emph{$k$-antichain} if it meets every chain in at most $k$ points.
Equivalently, $A$ contains no chain
\[
x^1\prec_C x^2\prec_C\cdots\prec_C x^{k+1}.
\]
By Mirsky's theorem~\cite[Theorem~2]{Mirsky}, every $k$-antichain can be partitioned into at most $k$ antichains.

When $C=(0,\infty)^n$, the relation $y-x\in C$ is equivalent to $x\ll y$, and the $1$-antichains are the weak antichains defined above.

Let $\ip{\cdot}{\cdot}$ and $\norm{\cdot}$ denote the Euclidean inner product and norm, respectively.
Define the \emph{dual cone} of $C$ by
\[
C^*=\{\xi\in\R^n:\ip{\xi}{z}\geq0
\text{ for every }z\in C\}.
\]
For $v\in C$, set $v^\perp=\{u\in\R^n:\ip{u}{v}=0\}$.
Let $P_v:\R^n\to v^\perp$ be the orthogonal projection
\[
P_v(x)=x-\frac{\ip{x}{v}}{\norm{v}^2}v.
\]
In particular, $P_v(x+tv)=P_v(x)$ for all $x\in\R^n$ and $t\in\R$.
Define
\begin{equation}\label{eq:alpha}
\alpha_C(v)=\min\{\ip{\xi}{v}:\xi\in C^*,\ \norm{\xi}=1\}.
\end{equation}
The minimum exists and is positive for every $v\in C$; see Lemma~\ref{lem:dual-cone}.

Our main result is the following.

\begin{theorem}\label{thm:cone}
Let $C\subsetneq\R^n$ be a nonempty open convex cone.
For $v\in C$ and a positive integer $k$, if $A\subseteq\R^n$ is a $k$-antichain with respect to $\prec_C$, then
\begin{equation}\label{eq:cone}
\Hh^{n-1}(A)
\leq
k\frac{\norm{v}}{\alpha_C(v)}
\Hh^{n-1}(P_v(A)).
\end{equation}
The constant $k\frac{\norm{v}}{\alpha_C(v)}$ is optimal for every $C$, $v$ and $k$.
\end{theorem}

For $C=(0,\infty)^n$, we have $C^*=[0,\infty)^n$ and $\alpha_C(v)=\min_{j\in[n]}v_j$.
Let $v=(v_1,\ldots,v_n)\in (0,\infty)^n$ and $i\in[n]$, define
\[
\Delta_{i,v}(x)
=\left(x_j-\frac{v_j}{v_i}x_i\right)_{j\neq i}.
\]
Then $\Delta_{i,v}(x+tv)=\Delta_{i,v}(x)$ for all $x\in\R^n$ and $t\in\R$.
Theorem~\ref{thm:cone} gives the following inequality.

\begin{corollary}\label{cor:weighted-projection}
Let $k$ be a positive integer, and let $A\subseteq\R^n$ be a $k$-antichain with respect to $\ll$.
For every $v\in(0,\infty)^n$ and every $i\in[n]$,
\[
\Hh^{n-1}(A)
\leq
k\frac{v_i}{\min_{j\in[n]}v_j}
\Hh^{n-1}(\Delta_{i,v}(A)).
\]
\end{corollary}

Taking $v=(1,\ldots,1)$ and $k=1$ gives the following.

\begin{corollary}\label{cor:individual}
For every weak antichain $A\subseteq\R^n$ and every $i\in[n]$,
\begin{equation}\label{eq:individual}
\Hh^{n-1}(A)
\leq
\Hh^{n-1}(\Delta_i(A)).
\end{equation}
\end{corollary}

Corollary~\ref{cor:individual} proves the conjecture of Engel, Mitsis, Pelekis and Reiher~\cite[Conjecture~5.3]{EMPR} for every $n\geq2$.
Indeed, if $A\subseteq[0,1]^n$ is a weak antichain, then the sets $\underline A_i$ defined in~\eqref{eq:pieces} cover $A$, and each is a weak antichain.
By subadditivity and~\eqref{eq:individual},
\[
\Hh^{n-1}(A)
\leq
\sum_{i=1}^n\Hh^{n-1}(\underline A_i)
\leq
\sum_{i=1}^n
\Hh^{n-1}\bigl(\Delta_i(\underline A_i)\bigr).
\]

We also recover the bound of Pelekis and Vlas\'ak~\cite{PelekisVlasak} for $k$-antichains in the unit cube.
Here we use the coordinatewise order
\[
x\preceq y
\quad\Longleftrightarrow\quad
x_i\leq y_i \text{ for every }i\in[n].
\]
Since every chain under $\ll$ is also a chain under $\preceq$, a $k$-antichain under $\preceq$ is a $k$-antichain under $\ll$.
For such a set $A\subseteq[0,1]^n$, apply Corollary~\ref{cor:weighted-projection} with $v=(1,\ldots,1)$ to each $\underline A_i$.
Using $\Delta_i(\underline A_i)\subseteq[0,1]^{n-1}$, we obtain
\[
\Hh^{n-1}(A)
\leq
\sum_{i=1}^n\Hh^{n-1}(\underline A_i)
\leq
k\sum_{i=1}^n
\Hh^{n-1}\bigl(\Delta_i(\underline A_i)\bigr)
\leq kn.
\]

Section~\ref{sec:preliminaries} establishes the Borel hull and cone geometry results used in the proof.
Section~\ref{sec:cone-proof} proves Theorem~\ref{thm:cone} and its projection corollaries.
Section~\ref{sec:cube} applies these inequalities to $k$-antichains in the unit cube.

\section{Borel hulls and the geometry of cone antichains}
\label{sec:preliminaries}

We first recall the definitions and standard results used in the proof.
For a point $u$ in a Euclidean space and $r>0$, let $B(u,r)$ denote the open ball of radius $r$ centered at $u$.

A \emph{$\sigma$-algebra} on $\R^m$ is a family of subsets of $\R^m$ that contains $\R^m$ and is closed under complements and countable unions~\cite[Definition~1.4]{EvansGariepy}.
The \emph{Borel $\sigma$-algebra} $\mathcal B(\R^m)$ is the intersection of all $\sigma$-algebras on $\R^m$ that contain every open subset of $\R^m$.
Its elements are called \emph{Borel sets};
see Evans and Gariepy~\cite[Definition~1.6]{EvansGariepy}.
In particular, every open or closed set is Borel, and Borel sets are closed under countable unions and countable intersections.

\begin{theorem}[{\cite[Theorem~2.1]{EvansGariepy}}]
\label{thm:borel-hull}
Let $V$ be an $m$-dimensional Euclidean space, where $m\geq1$. 
Every Borel subset of $V$ is $\Hh^m$-measurable.
For every $E\subseteq V$, there is a Borel set $M\subseteq V$ such that
\[
E\subseteq M,
\qquad
\Hh^m(M)=\Hh^m(E).
\]
\end{theorem}

We call such an $M$ an \emph{equal-measure Borel hull} of $E$.
The following separating hyperplane theorem will be used to establish the properties of the dual cone.

\begin{theorem}[{\cite[Section~2.5.1, p.~46]{BoydVandenberghe}}]
\label{thm:separation}
Let $D,E\subseteq\R^n$ be nonempty disjoint convex sets.
Then there exist a nonzero vector $\xi\in\R^n$ and a real number $a$ such that
\[
\ip{\xi}{x}\leq a\leq\ip{\xi}{y}
\qquad\text{for every }x\in D\text{ and }y\in E.
\]
\end{theorem}

\begin{lemma}
\label{lem:cone-separation}
Let $C\subsetneq\R^n$ be a nonempty open convex cone.
For every $d\in\R^n\setminus C$, there is a vector $\xi\in C^*$ such that
\[
\norm{\xi}=1,
\qquad
\ip{\xi}{d}\leq0.
\]
\end{lemma}

\begin{proof}
Let $d\notin C$.
The sets $\{d\}$ and $C$ are nonempty, disjoint and convex.
By Theorem~\ref{thm:separation}, there exist $\eta\in\R^n\setminus\{0\}$ and $a\in\R$ such that 
\[
\ip{\eta}{d}\leq a\leq\ip{\eta}{z}
\qquad\text{for every }z\in C.
\]
Since $C$ is a cone, $tz\in C$ for every $t>0$.
The separation inequality therefore gives $a\leq t\ip{\eta}{z}$ for every $t>0$.
If $\ip{\eta}{z}<0$, then $t\ip{\eta}{z}\to-\infty$ as $t\to\infty$, contradicting this inequality.
Thus $\ip{\eta}{z}\geq0$.
Since $z\in C$ was arbitrary, the definition of the dual cone gives $\eta\in C^*$.

Next, choose any fixed $z\in C$ and let $t\rightarrow0$ in $a\leq t\ip{\eta}{z}$.
This gives $a\leq0$, and hence $\ip{\eta}{d}\leq a\leq0$.
Taking $\xi=\frac{\eta}{\norm{\eta}}$ proves the assertion.
\end{proof}

A subset $K\subseteq\R^n$ is called \emph{compact} if every covering of $K$ by open sets has a finite subcover.
By the Heine--Borel theorem, a subset of $\R^n$ is compact if and only if it is closed and bounded.
By the extreme value theorem, every continuous real-valued function on a nonempty compact set attains its minimum and maximum.

\begin{lemma}\label{lem:dual-cone}
Let $C\subsetneq\R^n$ be a nonempty open convex cone.
Then $C^*$ contains a nonzero vector, and
\begin{equation}\label{eq:strict-duality}
\ip{\xi}{z}>0
\qquad\text{for every }z\in C
\text{ and every }\xi\in C^*\setminus\{0\}.
\end{equation}
Consequently, the minimum in~\eqref{eq:alpha} exists and 
\[
0<\alpha_C(v)\leq\norm{v}
\qquad\text{for every }v\in C.
\]
\end{lemma}

\begin{proof}
Since $C\subsetneq\R^n$, there exists $d\in\R^n\setminus C$.
By Lemma~\ref{lem:cone-separation}, $C^*$ contains a unit vector.
Fix $z\in C$ and $\xi\in C^*\setminus\{0\}$.
Since $C$ is open, $z-\varepsilon\xi\in C$ for all sufficiently small $\varepsilon>0$.
Choose such an $\varepsilon$.
The definition of $C^*$ gives
\[
0\leq\ip{\xi}{z-\varepsilon\xi}
=\ip{\xi}{z}-\varepsilon\norm{\xi}^2.
\]
Hence $\ip{\xi}{z}\geq\varepsilon\norm{\xi}^2>0$.
This proves~\eqref{eq:strict-duality}.

Set $K=\{\xi\in C^*:\norm{\xi}=1\}$.
The unit vector obtained above shows that $K$ is nonempty.
For each $z\in C$, the function $\xi\mapsto\ip{\xi}{z}$ is continuous, so $\{\xi\in\R^n:\ip{\xi}{z}\geq0\}$ is closed.
Therefore
\[
C^*=
\bigcap_{z\in C}
\{\xi\in\R^n:\ip{\xi}{z}\geq0\}
\]
is closed.
The unit sphere is also closed, so $K$ is closed.
Moreover, every vector in $K$ has norm one, so $K$ is bounded.
By the Heine--Borel theorem, $K$ is compact.

Now fix $v\in C$.
The function $\xi\mapsto\ip{\xi}{v}$ is continuous on the nonempty compact set $K$.
By the extreme value theorem, it attains its minimum at some $\xi_v\in K$.
Thus the minimum in~\eqref{eq:alpha} exists and $\alpha_C(v)=\ip{\xi_v}{v}$.
Since $\xi_v\in C^*\setminus\{0\}$ and $v\in C$, equation~\eqref{eq:strict-duality} gives $\alpha_C(v)>0$.
Finally, the Cauchy--Schwarz inequality and
$\norm{\xi_v}=1$ give
\[
\alpha_C(v)=\ip{\xi_v}{v}
\leq\norm{\xi_v}\norm{v}
=\norm{v}.
\]
This completes the proof.
\end{proof}

Let $C\subsetneq\R^n$ be a nonempty open convex cone.
Fix $v\in C$ and set $\widehat v=\frac{v}{\norm{v}}$.
Let $A\subseteq \R^n$ be an antichain with respect to $\prec_C$, and put $E=P_v(A)$.
We show that the restriction $P_v|_A$ is injective.
Suppose, to the contrary, that there exist distinct points $x,y\in A$ such that $P_v(x)=P_v(y)$.
Since $P_v$ is the orthogonal projection onto $v^\perp$, its kernal is the line spanned by $v$.
Thus $x-y=\lambda v$ for some nonzero scalar $\lambda\in\R$.
Since $C$ is a cone containing $v$, we have $x-y\in C$ if $\lambda>0$, and $y-x\in C$ if $\lambda<0$.
This contradicts the fact that $A$ is an antichain with respect to $\prec_C$.

Since $P_v|_A$ is injective and $E=P_v(A)$, each $u\in E$ has a unique preimage $x\in A$ satisfying $P_v(x)=u$.
Define $f:E\to \R$ by 
\begin{equation}\label{eq:graph-function}
f(u)=\ip{x}{\widehat v}
\qquad(u\in E).
\end{equation}
By the definition of $P_v$, $x=P_v(x)+\ip{x}{\widehat v}\widehat v$, where $P_v(x)$ is perpendicular to $v$ and $\ip{x}{\widehat v}\widehat v$ is parallel to $v$.
Since $P_v(x)=u$ and $f(u)=\ip{x}{\widehat v}$, this gives $x=u+f(u)\widehat v$.
Define $F:E\to\R^n$ by $F(u)=u+f(u)\widehat v$.
Then
\begin{equation}\label{eq:graph}
A=F(E)=\{u+f(u)\widehat v:u\in E\}.
\end{equation}
The uniqueness of the preimage $x$ shows that $f$ is uniquely determined by $A$ for the fixed vector $v$.

Let $D\subseteq v^\perp$.
A function $\varphi:D\to\R$ is called \emph{Lipschitz} if there exists a constant $L\geq0$ such that
\[
|f(u)-f(w)|\leq L\norm{u-w}
\qquad\text{for every }u,w\in D.
\]
Such a function is also called \emph{$L$-Lipschitz}; see~\cite[Definition~3.1]{EvansGariepy}.

\begin{lemma}\label{lem:graph}
The function $f:E\to\R$ defined by~\eqref{eq:graph-function} satisfies
\begin{equation}\label{eq:graph-lipschitz}
|f(w)-f(u)|
\leq
\sqrt{\frac{\norm{v}^2}{\alpha_C(v)^2}-1}\,
\norm{w-u}
\qquad\text{for every }u,w\in E.
\end{equation}
In particular, $f$ is Lipschitz.
\end{lemma}

\begin{proof}
Fix $u,w\in E$ and $v\in C$.
Suppose first that $f(w)>f(u)$.
By the definition of $F$,
\[
F(w)-F(u)
=(w-u)+\bigl(f(w)-f(u)\bigr)\widehat v.
\]
Since $F(u),F(w)\in A$ and $A$ is an antichain with respect to $\prec_C$, this difference does not belong to $C$.
By Lemma~\ref{lem:cone-separation}, there exists $\xi\in C^*$ with $\norm{\xi}=1$ such that
\begin{equation}\label{eq:graph-separation}
\ip{\xi}{(w-u)+\bigl(f(w)-f(u)\bigr)\widehat v}
\leq0.
\end{equation}
Since $\xi$ is a unit vector in $C^*$, the definition of $\alpha_C(v)$ gives $\ip{\xi}{v}\geq\alpha_C(v)$.
Consequently,
\[
\ip{\xi}{\widehat v}
=\frac{\ip{\xi}{v}}{\norm{v}}
\geq\frac{\alpha_C(v)}{\norm{v}}>0,
\]
where positivity follows from Lemma~\ref{lem:dual-cone}.

Since $\norm{\widehat v}=1$, the decomposition $\xi=P_v(\xi)+\ip{\xi}{\widehat v}\widehat v$ and the Pythagorean theorem give
\[
\norm{\xi}^2=\norm{P_v(\xi)}^2+\ip{\xi}{\widehat v}^2,
\]
Together with $\norm{\xi}=1$, this yields $\norm{P_v(\xi)}^2=1-\ip{\xi}{\widehat v}^2$.
Moreover, $w-u\in v^\perp$ implies $\ip{\widehat v}{w-u}=0$.
Taking the inner product with $w-u$ in the decomposition $\xi=P_v(\xi)+\ip{\xi}{\widehat v}\widehat v$ gives
\begin{equation}\label{eq:graph-inner-product}
\ip{\xi}{w-u}=\ip{P_v(\xi)}{w-u}.
\end{equation}
Expanding~\eqref{eq:graph-separation} gives
\[
\ip{\xi}{w-u}
+\bigl(f(w)-f(u)\bigr)\ip{\xi}{\widehat v}
\leq0.
\]
Substituting ~\eqref{eq:graph-inner-product} into this inequality and rearranging, we obtain
\[
\bigl(f(w)-f(u)\bigr)\ip{\xi}{\widehat v}
\leq-\ip{P_v(\xi)}{w-u}.
\]

By the Cauchy--Schwarz inequality, $-\ip{P_v(\xi)}{w-u} \leq|\ip{P_v(\xi)}{w-u}| \leq\norm{P_v(\xi)}\norm{w-u}$.
Since $\ip{\xi}{\widehat v}>0$, division gives
\[
\begin{aligned}
f(w)-f(u)
&\leq
\frac{\norm{P_v(\xi)}}{\ip{\xi}{\widehat v}}\,
\norm{w-u}\\
&=
\sqrt{\frac{1}{\ip{\xi}{\widehat v}^2}-1}\,
\norm{w-u}\\
&\leq
\sqrt{\frac{\norm{v}^2}{\alpha_C(v)^2}-1}\,
\norm{w-u}.
\end{aligned}
\]
The last inequality follows from $\ip{\xi}{\widehat v}\geq\frac{\alpha_C(v)}{\norm{v}}>0$.

If $f(w)<f(u)$, apply the same argument with $u$ and $w$ interchanged.
This gives the same upper bound for $f(u)-f(w)$, since $\norm{u-w}=\norm{w-u}$.
If $f(w)=f(u)$, the desired inequality is immediate.
Thus~\eqref{eq:graph-lipschitz} holds for every $u,w\in E$, and $f$ is Lipschitz.
\end{proof}

We use the following Lipschitz extension theorem due to McShane~\cite{McShane}.

\begin{theorem}[\cite{McShane}]
\label{thm:mcshane}
Let $V$ be a Euclidean space and let $E\subseteq V$.
Suppose that $f:E\to\R$ is Lipschitz with constant $L\geq 0$. 
Then there exists a function $\widetilde f:V\to\R$ such that
$\widetilde f=f$ on $E$ and
\[
|\widetilde f(u)-\widetilde f(w)|
\leq L\norm{u-w}
\qquad\text{for every }u,w\in V.
\]
\end{theorem}

Set
\begin{equation}\label{eq:graph-constant}
L=\sqrt{\frac{\norm{v}^2}{\alpha_C(v)^2}-1}.
\end{equation}
By Lemma~\ref{lem:graph}, the function $f:E\to\R$ is $L$-Lipschitz.
Applying Theorem~\ref{thm:mcshane} with $V=v^\perp$ gives a function $\widetilde f:v^\perp\to\R$ such that $\widetilde f=f$ on $E$ and 
\begin{equation}\label{eq:extension-lipschitz}
|\widetilde f(u)-\widetilde f(w)|
\leq L\norm{u-w}
\qquad\text{for every }u,w\in v^\perp.
\end{equation}

Define $\widetilde F:v^\perp\to\R^n$ by
\begin{equation}\label{eq:extended-graph}
\widetilde F(u)=u+\widetilde f(u)\widehat v
\qquad(u\in v^\perp).
\end{equation}
For every $u\in E$, the identity $\widetilde f(u)=f(u)$ gives $\widetilde F(u)=F(u)$.
Hence
\[
A=\widetilde F(E).
\]
Moreover, $P_v(\widetilde F(u))=u$ for every $u\in v^\perp$.
Thus $\widetilde F(u)=\widetilde F(w)$ implies $u=w$, so $\widetilde F$ is injective.
For $u,w\in v^\perp$, we have
\[
\widetilde F(u)-\widetilde F(w)
=(u-w)+\bigl(\widetilde f(u)-\widetilde f(w)\bigr)\widehat v.
\]
The two terms on the right are perpendicular, and $\norm{\widehat v}=1$.
The Pythagorean theorem and~\eqref{eq:extension-lipschitz} therefore give
\[
\begin{aligned}
\norm{\widetilde F(u)-\widetilde F(w)}^2
&=\norm{u-w}^2
  +|\widetilde f(u)-\widetilde f(w)|^2\\
&\leq(1+L^2)\norm{u-w}^2.
\end{aligned}
\]
Equation~\eqref{eq:graph-constant} gives $\sqrt{1+L^2}=\frac{\norm{v}}{\alpha_C(v)}$.
Taking square roots, we obtain
\begin{equation}\label{eq:extended-graph-lipschitz}
\norm{\widetilde F(u)-\widetilde F(w)}
\leq
\frac{\norm{v}}{\alpha_C(v)}\norm{u-w}
\qquad\text{for every }u,w\in v^\perp.
\end{equation}
In particular, $\widetilde F$ is Lipschitz.

Let $V$ and $W$ be Euclidean spaces.
A map $G:V\to W$ is called \emph{differentiable} at $u\in V$ if there exists a linear map $DG(u):V\to W$ such that
\[
\lim_{\substack{h\to0\\h\in V\setminus\{0\}}}
\frac{\norm{G(u+h)-G(u)-DG(u)[h]}}{\norm{h}}=0.
\]
The linear map $DG(u)$ is called the \emph{derivative} of $G$ at $u$; see~\cite[Definition~3.2]{EvansGariepy}.
For a real-valued function $G:V\to\R$, the derivative is represented by a unique vector $\nabla G(u)\in V$ such that
\[
DG(u)[h]=\ip{\nabla G(u)}{h}
\qquad\text{for every }h\in V.
\]
This vector is called the \emph{gradient} of $G$ at $u$.

A property holds $\Hh^m$-almost everywhere if the set of points where it fails has $\Hh^m$-measure zero.

\begin{theorem}[{\cite[Theorem~3.2]{EvansGariepy}}]
\label{thm:rademacher}
Let $V$ and $W$ be Euclidean spaces, with $\dim V=m$.
If $G:V\to W$ is Lipschitz, then $G$ is differentiable at $\Hh^m$-almost every point of $V$.
\end{theorem}

By~\eqref{eq:extension-lipschitz}, the function $\widetilde f:v^\perp\to\R$ is Lipschitz.
Theorem~\ref{thm:rademacher} therefore implies that $\widetilde f$ is differentiable at $\Hh^{n-1}$-almost every point of $v^\perp$.

Fix a point $u\in v^\perp$ at which $\widetilde f$ is differentiable.
For every unit vector $h\in v^\perp$, equation~\eqref{eq:extension-lipschitz} gives
\[
\left|
\frac{\widetilde f(u+th)-\widetilde f(u)}{t}
\right|
\leq L
\qquad(t\in\R\setminus\{0\}).
\]
Letting $t\to0$ and using the definition of the gradient,
we obtain
\[
|\ip{\nabla\widetilde f(u)}{h}|\leq L.
\]
If $\nabla\widetilde f(u)\neq0$, choose $h=\frac{\nabla\widetilde f(u)}{\norm{\nabla\widetilde f(u)}}$.
This gives $\norm{\nabla\widetilde f(u)}\leq L$.
The same inequality holds when $\nabla\widetilde f(u)=0$.
Therefore
\begin{equation}\label{eq:gradient-bound}
\norm{\nabla\widetilde f(u)}
\leq
L=\sqrt{\frac{\norm{v}^2}{\alpha_C(v)^2}-1}
\end{equation}
at every point where $\widetilde f$ is differentiable.

By~\eqref{eq:extended-graph}, for every $h\in v^\perp$,
\[
\widetilde F(u+h)-\widetilde F(u)
=h+\bigl(\widetilde f(u+h)-\widetilde f(u)\bigr)\widehat v.
\]
Differentiability of $\widetilde f$ at $u$ gives
\begin{equation}\label{eq:extension-differentiability}
\lim_{\substack{h\to0\\h\in v^\perp\setminus\{0\}}}
\frac{
\left|
\widetilde f(u+h)-\widetilde f(u)
-\ip{\nabla\widetilde f(u)}{h}
\right|
}{\norm{h}}=0.
\end{equation}
Equation~\eqref{eq:extended-graph} and $\norm{\widehat v}=1$ give
\[
\norm{\widetilde F(u+h)-\widetilde F(u)
-h-\ip{\nabla\widetilde f(u)}{h}\widehat v}
=
\left|\widetilde f(u+h)-\widetilde f(u)
-\ip{\nabla\widetilde f(u)}{h}\right|.
\]
Dividing both sides by $\norm{h}$ and using~\eqref{eq:extension-differentiability} shows that the quotient on the left tends to zero as $h\to0$.
The map $h\mapsto h+\ip{\nabla\widetilde f(u)}{h}\widehat v$ is linear.
Thus, by the definition of differentiability, $\widetilde F$ is differentiable at $u$, with derivative
\begin{equation}\label{eq:graph-derivative}
D\widetilde F(u)[h]
=h+\ip{\nabla\widetilde f(u)}{h}\widehat v
\qquad(h\in v^\perp).
\end{equation}

Let $V$ and $W$ be Euclidean spaces with $\dim V=m\leq\dim W$, and let $T:V\to W$ be linear.
Choose an orthonormal basis $e_1,\ldots,e_m$ of $V$.
The \emph{$m$-dimensional Jacobian} of $T$ is
\[
J_m(T)
=
\sqrt{\det\bigl[\ip{T(e_i)}{T(e_j)}\bigr]_{i,j=1}^{m}}.
\]
This value does not depend on the choice of orthonormal basis; see~\cite[Theorem~3.6]{EvansGariepy}.

For a Lipschitz map $G:V\to W$, define
\begin{equation}\label{eq:jacobian-definition}
J_mG(u)=J_m(DG(u))
\end{equation}
at every point $u$ where $G$ is differentiable, and set $J_mG(u)=0$ at the remaining points; see~\cite[Definition~3.6]{EvansGariepy}.

\begin{lemma}\label{lem:graph-jacobian}
At every point $u\in v^\perp$ where $\widetilde f$ is differentiable,
\begin{equation}\label{eq:Jacobian}
J_{n-1}\widetilde F(u)
=
\sqrt{1+\norm{\nabla\widetilde f(u)}^2}
\leq
\frac{\norm{v}}{\alpha_C(v)}.
\end{equation}
In particular, this bound holds at $\Hh^{n-1}$-almost every point of $v^\perp$.
\end{lemma}

\begin{proof}
Fix a point $u$ where $\widetilde f$ is differentiable.
If $\nabla\widetilde f(u)=0$, equation~\eqref{eq:graph-derivative} gives $D\widetilde F(u)[h]=h$ for every $h\in v^\perp$, so $J_{n-1}\widetilde F(u)=1$.

Suppose that $\nabla\widetilde f(u)\neq0$.
Choose an orthonormal basis $e_1,\ldots,e_{n-1}$ of $v^\perp$ with $e_1=\frac{\nabla\widetilde f(u)}{\norm{\nabla\widetilde f(u)}}$.
By~\eqref{eq:graph-derivative},
\[
D\widetilde F(u)[e_1]
=e_1+\norm{\nabla\widetilde f(u)}\widehat v,
\qquad
D\widetilde F(u)[e_j]=e_j
\quad(2\leq j\leq n-1).
\]
These vectors are mutually perpendicular.
Thus the matrix defining the Jacobian is diagonal, and its determinant is the product of their squared norms.
The first has squared norm $1+\norm{\nabla\widetilde f(u)}^2$, and each of the others has squared norm one.
Consequently,
\[
J_{n-1}\widetilde F(u)
=\sqrt{1+\norm{\nabla\widetilde f(u)}^2}.
\]
This formula also holds when $\nabla\widetilde f(u)=0$.

Finally, equations~\eqref{eq:gradient-bound}
and~\eqref{eq:graph-constant} give
\[
J_{n-1}\widetilde F(u)
=\sqrt{1+\norm{\nabla\widetilde f(u)}^2}
\leq\sqrt{1+L^2}
=\frac{\norm{v}}{\alpha_C(v)}.
\]
The almost-everywhere assertion follows from Theorem~\ref{thm:rademacher}.
\end{proof}

An injective linear map changes Hausdorff outer measure by the factor given by its Jacobian.

\begin{theorem}[{\cite[Lemma~3.1]{EvansGariepy}}]
\label{thm:linear-measure}
Let $V$ and $W$ be Euclidean spaces with $1\leq m=\dim V\leq\dim W$.
Let $T:V\to W$ be an injective linear map.
Then, for every set $E\subseteq V$,
\begin{equation}\label{eq:linear-measure}
\Hh^m(T(E))
=
J_m(T)\Hh^m(E).
\end{equation}
\end{theorem}

For injective Lipschitz maps, the area formula takes the following form.

\begin{theorem}[{\cite[Theorem~3.8]{EvansGariepy}}]
\label{thm:area-formula}
Let $V$ and $W$ be Euclidean spaces with $1\leq m=\dim V\leq\dim W$.
Let $G:V\to W$ be an injective Lipschitz map.
Then, for every Borel set $M\subseteq V$,
\begin{equation}\label{eq:area-formula}
\Hh^m(G(M))
=
\int_M J_mG(u)\,d\Hh^m(u).
\end{equation}
\end{theorem}

\section{Cone projection inequalities}
\label{sec:cone-proof}

For sets $S,T\subseteq\R^n$, define
\[
S+T=\{x+y:x\in S,\ y\in T\}.
\]
For a vector $w\in\R^n$, write $S+w=S+\{w\}=\{x+w:x\in S\}$.
This is the translate of $S$ by $w$.

\begin{proof}[Proof of Theorem~\ref{thm:cone}]
Let $C\subsetneq\R^n$ be a nonempty open convex cone.
Fix $v\in C$, set $\widehat v=\frac{v}{\norm{v}}$, and let $A\subseteq\R^n$ be a $k$-antichain with respect to $\prec_C$.
The inequality is immediate if $A=\varnothing$ or $\Hh^{n-1}(P_v(A))=\infty$.
Now we assume that $A\neq\varnothing$ and $\Hh^{n-1}(P_v(A))<\infty$.

We first consider that $k=1$, and set $E=P_v(A)$.
The graph representation and extension constructed in Section~\ref{sec:preliminaries} give an injective Lipschitz map $\widetilde F:v^\perp\to\R^n$ such that $A=\widetilde F(E)$.
By Lemma~\ref{lem:graph-jacobian}, $J_{n-1}\widetilde F(u)\leq\frac{\norm{v}}{\alpha_C(v)}$ at $\Hh^{n-1}$-almost every point of $v^\perp$.

By Theorem~\ref{thm:borel-hull}, there exists a Borel set $M\subseteq v^\perp$ containing $E$ such that $\Hh^{n-1}(M)=\Hh^{n-1}(E)$.
The inclusion $E\subseteq M$ gives $A\subseteq\widetilde F(M)$.
Applying Theorem~\ref{thm:area-formula}, we obtain
\begin{equation}\label{eq:cone-antichain}
\begin{aligned}
\Hh^{n-1}(A)
&\leq\Hh^{n-1}(\widetilde F(M))\\
&=\int_M J_{n-1}\widetilde F(u)\,d\Hh^{n-1}(u)\\
&\leq\frac{\norm{v}}{\alpha_C(v)}\Hh^{n-1}(M)\\
&=\frac{\norm{v}}{\alpha_C(v)}\Hh^{n-1}(P_v(A)).
\end{aligned}
\end{equation}

Now let $k$ be arbitrary.
For each $x\in A$, let $r(x)$ be the largest integer $\ell$ for which there is a chain
\[
x^1\prec_C\cdots\prec_C x^\ell=x
\]
of points in $A$.
The one-point chain shows that $r(x)\geq1$, and the $k$-antichain assumption gives $r(x)\leq k$.

For $j\in\{1,\ldots,k\}$, set
\[
A_j=\{x\in A:r(x)=j\}.
\]
Each $A_j$ is an antichain with respect to $\prec_C$.
Indeed, if $x,y\in A_j$ satisfied $x\prec_C y$, then $r(y)\geq r(x)+1=j+1$, contradicting $r(y)=j$.
These sets partition $A$, and $P_v(A_j)\subseteq P_v(A)$.
By subadditivity and~\eqref{eq:cone-antichain},
\[
\begin{aligned}
\Hh^{n-1}(A)
&\leq\sum_{j=1}^k\Hh^{n-1}(A_j)\\
&\leq\frac{\norm{v}}{\alpha_C(v)}
       \sum_{j=1}^k\Hh^{n-1}(P_v(A_j))\\
&\leq k\frac{\norm{v}}{\alpha_C(v)}
       \Hh^{n-1}(P_v(A)).
\end{aligned}
\]
This proves~\eqref{eq:cone}.

To prove optimality, choose a unit vector $\xi\in C^*$ such that $\ip{\xi}{v}=\alpha_C(v)$.
Such a vector exists by Lemma~\ref{lem:dual-cone}, which also gives $\alpha_C(v)>0$.
Recall that $\alpha_C(v)>0$, and set $\widehat v=\frac{v}{\norm{v}}$.

Define the linear map $G:v^\perp\to\R^n$ by
\begin{equation}\label{eq:extremal-map}
G(u)=u-\frac{\ip{\xi}{u}}{\alpha_C(v)}v
\qquad(u\in v^\perp).
\end{equation}
For $u\in v^\perp$, we have $P_v(u)=u$ and $P_v(v)=0$.
Thus
\[
P_v(G(u))
=P_v(u)-\frac{\ip{\xi}{u}}{\alpha_C(v)}P_v(v)
=u.
\]
Also, $\ip{\xi}{v}=\alpha_C(v)$ gives
\[
\ip{\xi}{G(u)}
=\ip{\xi}{u}
-\frac{\ip{\xi}{u}}{\alpha_C(v)}\ip{\xi}{v}
=0.
\]
Therefore
\begin{equation}\label{eq:extremal-map-identities}
P_v(G(u))=u,
\qquad
\ip{\xi}{G(u)}=0
\qquad(u\in v^\perp).
\end{equation}
In particular, $G$ is injective: if $G(u)=G(w)$, applying $P_v$ gives $u=w$.

To compute the Jacobian of $G$, define $\varphi:v^\perp\to\R$ by
\[
\varphi(u)
=-\frac{\norm{v}}{\alpha_C(v)}\ip{\xi}{u}.
\]
The identity $v=\norm{v}\widehat v$ gives $G(u)=u+\varphi(u)\widehat v$.
For every $u,h\in v^\perp$,
\[
\varphi(u+h)-\varphi(u)
=-\frac{\norm{v}}{\alpha_C(v)}\ip{\xi}{h}.
\]
The vector $\xi-P_v(\xi)$ is parallel to $v$, whereas $h$ is perpendicular to $v$.
Hence $\ip{\xi}{h}=\ip{P_v(\xi)}{h}$, and
\[
\varphi(u+h)-\varphi(u)
=
\ip{-\frac{\norm{v}}{\alpha_C(v)}P_v(\xi)}{h}.
\]
This difference is exactly linear in $h$.
By the definitions of differentiability and the gradient,
\begin{equation}\label{eq:extremal-gradient}
\nabla\varphi(u)
=-\frac{\norm{v}}{\alpha_C(v)}P_v(\xi)
\qquad(u\in v^\perp).
\end{equation}

The orthogonal decomposition $\xi=P_v(\xi)+\ip{\xi}{\widehat v}\widehat v$ and $\norm{\xi}=\norm{\widehat v}=1$ give
\[
1=\norm{P_v(\xi)}^2+\ip{\xi}{\widehat v}^2.
\]
Moreover, $\ip{\xi}{\widehat v}=\frac{\alpha_C(v)}{\norm{v}}$.
Consequently,
\begin{equation}\label{eq:extremal-projection-norm}
\norm{P_v(\xi)}^2
=1-\frac{\alpha_C(v)^2}{\norm{v}^2}.
\end{equation}

The map $G$ is linear, so $DG(u)=G$ at every $u\in v^\perp$.
Its graph representation gives
\[
DG(u)[h]
=h+\ip{\nabla\varphi(u)}{h}\widehat v
\qquad(h\in v^\perp).
\]
This has the same form as the derivative used in the proof of Lemma~\ref{lem:graph-jacobian}.
The Jacobian calculation there, together with~\eqref{eq:extremal-gradient} and~\eqref{eq:extremal-projection-norm}, gives
\begin{equation}\label{eq:extremal-jacobian}
\begin{aligned}
J_{n-1}(G)
&=\sqrt{1+\norm{\nabla\varphi(u)}^2}\\
&=\sqrt{
1+\frac{\norm{v}^2}{\alpha_C(v)^2}\norm{P_v(\xi)}^2
}\\
&=\sqrt{
1+\frac{\norm{v}^2}{\alpha_C(v)^2}
\left(1-\frac{\alpha_C(v)^2}{\norm{v}^2}\right)
}\\
&=\frac{\norm{v}}{\alpha_C(v)}.
\end{aligned}
\end{equation}

Let
\[
D=\{u\in v^\perp:\norm{u}\leq1\}.
\]
This is the closed unit ball in the $(n-1)$-dimensional
Euclidean space $v^\perp$.
It is compact, and
$\Hh^{n-1}(D)=\omega_{n-1}\in(0,\infty)$.

For $j\in\{1,\ldots,k\}$, define
\[
B_j=G(D)+jv=\{G(u)+jv:u\in D\},
\qquad
B=\bigcup_{j=1}^k B_j.
\]
If $x\in B_j$, then $x=G(u)+jv$ for some $u\in D$.
By~\eqref{eq:extremal-map-identities},
\begin{equation}\label{eq:extremal-levels}
\ip{\xi}{x}
=\ip{\xi}{G(u)}+j\ip{\xi}{v}
=j\alpha_C(v).
\end{equation}
If $x\in B_i\cap B_j$ for $1\leq i,j\leq k$,
then~\eqref{eq:extremal-levels} gives
$i\alpha_C(v)=j\alpha_C(v)$.
Because $\alpha_C(v)>0$, this implies $i=j$.
Thus the sets $B_1,\ldots,B_k$ are pairwise disjoint.

Now suppose that $x,y\in B$ satisfy $x\prec_C y$.
Then $y-x\in C$, so~\eqref{eq:strict-duality} gives
\[
\ip{\xi}{y}-\ip{\xi}{x}
=\ip{\xi}{y-x}>0.
\]
The values of $\ip{\xi}{x}$ therefore strictly increase
along every chain in $B$.
By~\eqref{eq:extremal-levels}, these values belong to
$\{\alpha_C(v),2\alpha_C(v),\ldots,k\alpha_C(v)\}$.
Every chain in $B$ consequently has at most $k$ points.
Hence $B$ is a $k$-antichain with respect to $\prec_C$.

For every $u\in D$ and every $j\in\{1,\ldots,k\}$,
\[
P_v(G(u)+jv)=P_v(G(u))+jP_v(v)=u.
\]
As $u$ ranges over $D$, this shows that $P_v(B)=D$.

The map $G$ is continuous because it is linear.
Thus $G(D)$ and each of its translates $B_j$ are compact.
In particular, the sets $B_j$ are Borel and $\Hh^{n-1}$-measurable.
Their disjointness and translation invariance give
\[
\Hh^{n-1}(B)
=\sum_{j=1}^k\Hh^{n-1}(B_j)
=k\Hh^{n-1}(G(D)).
\]
Applying Theorem~\ref{thm:linear-measure} to the injective linear map $G$ and using~\eqref{eq:extremal-jacobian}, we obtain
\[
\Hh^{n-1}(G(D))
=J_{n-1}(G)\Hh^{n-1}(D)
=\frac{\norm{v}}{\alpha_C(v)}\Hh^{n-1}(D).
\]
Combining these identities with $P_v(B)=D$ yields
\[
\Hh^{n-1}(B)
=k\frac{\norm{v}}{\alpha_C(v)}
  \Hh^{n-1}(P_v(B)).
\]
This equality shows that the constant $k\norm{v}/\alpha_C(v)$ is optimal.
\end{proof}

Let $e_1,\ldots,e_n$ denote the standard basis of $\R^n$.
For $C=(0,\infty)^n$, the dual cone is $C^*=[0,\infty)^n$.
For every $v\in(0,\infty)^n$,
\begin{equation}\label{eq:orthant-alpha}
\alpha_C(v)=\min_{j\in[n]}v_j.
\end{equation}
Indeed, every unit vector $\xi\in[0,\infty)^n$ satisfies
\[
\ip{\xi}{v}
=\sum_{j=1}^n\xi_jv_j
\geq
\left(\min_{j\in[n]}v_j\right)\sum_{j=1}^n\xi_j
\geq\min_{j\in[n]}v_j,
\]
where the last inequality follows from $\sum_{j=1}^n\xi_j\geq\norm{\xi}=1$.
Equality holds for $\xi=e_j$ whenever $v_j=\min_{\ell\in[n]}v_\ell$.

Recall that the weighted skewed projection is defined by
\[
\Delta_{i,v}(x)
=\left(x_j-\frac{v_j}{v_i}x_i\right)_{j\neq i}.
\]

\begin{lemma}\label{lem:skewed-measure}
Let $v\in(0,\infty)^n$ and $i\in[n]$.
For every set $E\subseteq\R^n$,
\begin{equation}\label{eq:skewed-measure}
\Hh^{n-1}(\Delta_{i,v}(E))
=
\frac{\norm{v}}{v_i}
\Hh^{n-1}(P_v(E)).
\end{equation}
\end{lemma}

\begin{proof}
Let $T:v^\perp\to\R^{n-1}$ be the restriction of $\Delta_{i,v}$ to $v^\perp$.
The map $T$ is linear.
To prove injectivity, let $u,w\in v^\perp$ satisfy $T(u)=T(w)$.
Then $u_j-w_j=(\frac{v_j}{v_i})(u_i-w_i)$ for every $j\neq i$, and hence
\[
u-w=\frac{u_i-w_i}{v_i}v.
\]
Taking the inner product with $v$ and using $u-w\in v^\perp$ gives
\[
0=\ip{u-w}{v}
=\frac{u_i-w_i}{v_i}\norm{v}^2.
\]
Thus $u=w$.
Therefore $T$ is injective.

For $u,w\in v^\perp$, the vectors $u-(\frac{u_i}{v_i})v$ and $w-(\frac{w_i}{v_i})v$ have zero $i$th coordinate.
Their remaining coordinates are $T(u)$ and $T(w)$, respectively.
Therefore
\begin{equation}\label{eq:skewed-inner-product}
\begin{aligned}
\ip{T(u)}{T(w)}
&=
\ip{u-\frac{u_i}{v_i}v}
   {w-\frac{w_i}{v_i}v}\\
&=
\ip{u}{w}
+\frac{\norm{v}^2}{v_i^2}u_iw_i,
\end{aligned}
\end{equation}
where the last equality uses $\ip{u}{v}=\ip{w}{v}=0$.

The definition of $P_v$ gives
\[
\norm{P_v(e_i)}^2
=1-\frac{v_i^2}{\norm{v}^2}>0.
\]
Choose an orthonormal basis $b_1,\ldots,b_{n-1}$ of $v^\perp$ with
\[
b_1=\frac{P_v(e_i)}{\norm{P_v(e_i)}}.
\]
For every $u\in v^\perp$, $u_i=\ip{e_i}{u}=\ip{P_v(e_i)}{u}$.
Consequently,
\[
(b_1)_i=\norm{P_v(e_i)},
\qquad
(b_\ell)_i=0
\quad(2\leq\ell\leq n-1).
\]
Equation~\eqref{eq:skewed-inner-product} now shows that $T(b_1),\ldots,T(b_{n-1})$ are mutually perpendicular.
Moreover,
\[
\begin{aligned}
\norm{T(b_1)}^2
&=
1+\frac{\norm{v}^2}{v_i^2}\norm{P_v(e_i)}^2\\
&=
1+\frac{\norm{v}^2}{v_i^2}
\left(1-\frac{v_i^2}{\norm{v}^2}\right)
=\frac{\norm{v}^2}{v_i^2},
\end{aligned}
\]
and $\norm{T(b_\ell)}^2=1$ for $2\leq\ell\leq n-1$.
Thus the matrix defining the Jacobian is diagonal, and
\begin{equation}\label{eq:skewed-jacobian}
J_{n-1}(T)=\frac{\norm{v}}{v_i}.
\end{equation}

For every $x\in\R^n$, the difference $x-P_v(x)$ is parallel to $v$.
Since $\Delta_{i,v}(x+tv)=\Delta_{i,v}(x)$ for every $t\in\R$, we have $\Delta_{i,v}(x)=T(P_v(x))$.
Hence
\[
\Delta_{i,v}(E)=T(P_v(E)).
\]
Applying Theorem~\ref{thm:linear-measure} and~\eqref{eq:skewed-jacobian} gives
\[
\Hh^{n-1}(\Delta_{i,v}(E))
=
J_{n-1}(T)\Hh^{n-1}(P_v(E))
=
\frac{\norm{v}}{v_i}\Hh^{n-1}(P_v(E)).
\]
This proves~\eqref{eq:skewed-measure}.
\end{proof}

\begin{proof}[Proof of Corollary~\ref{cor:weighted-projection}]
Let $A\subseteq\R^n$ be a $k$-antichain with respect to $\ll$.
For $C=(0,\infty)^n$, the relation $\prec_C$ is exactly $\ll$.
Fix $v\in(0,\infty)^n$ and $i\in[n]$.
Theorem~\ref{thm:cone}, equation~\eqref{eq:orthant-alpha} and Lemma~\ref{lem:skewed-measure} give
\[
\begin{aligned}
\Hh^{n-1}(A)
&\leq
k\frac{\norm{v}}{\min_{j\in[n]}v_j}
\Hh^{n-1}(P_v(A))\\
&=
k\frac{\norm{v}}{\min_{j\in[n]}v_j}
\frac{v_i}{\norm{v}}
\Hh^{n-1}(\Delta_{i,v}(A))\\
&=
k\frac{v_i}{\min_{j\in[n]}v_j}
\Hh^{n-1}(\Delta_{i,v}(A)).
\end{aligned}
\]
\end{proof}

\begin{proof}[Proof of Corollary~\ref{cor:individual}]
A weak antichain is an antichain with respect to $\ll$.
Taking $k=1$ and $v=(1,\ldots,1)$ in Corollary~\ref{cor:weighted-projection},
we have $\Delta_{i,v}=\Delta_i$ and $v_i/\min_{j\in[n]}v_j=1$.
Thus
\[
\Hh^{n-1}(A)
\leq\Hh^{n-1}(\Delta_i(A))
\qquad\text{for every }i\in[n].
\]
\end{proof}

\section{Antichains in the unit cube}
\label{sec:cube}

Let $A\subseteq[0,1]^n$, and let $\underline A_i$ be the sets defined in~\eqref{eq:pieces}.
Every point of $A$ has at least one minimum coordinate, so
\[
A=\bigcup_{i=1}^n\underline A_i.
\]
For $x\in\underline A_i$, we have $0\leq x_j-x_i\leq1$ for every $j\neq i$.
Therefore
\[
\Delta_i(\underline A_i)\subseteq[0,1]^{n-1},
\]
and hence
\begin{equation}\label{eq:cube-projection-bound}
\Hh^{n-1}\bigl(\Delta_i(\underline A_i)\bigr)
\leq1
\qquad(i\in[n]).
\end{equation}

Corollary~\ref{cor:individual} now proves the skewed projection conjecture of Engel, Mitsis, Pelekis and Reiher~\cite[Conjecture~5.3]{EMPR}.

\begin{proof}[Proof of Conjecture~\ref{conj:skewed-projection}]
Let $A\subseteq[0,1]^n$ be a weak antichain.
Each $\underline A_i$ is a subset of $A$, so it is also a weak antichain.
By subadditivity and Corollary~\ref{cor:individual},
\[
\begin{aligned}
\Hh^{n-1}(A)
&\leq
\sum_{i=1}^n\Hh^{n-1}(\underline A_i)\\
&\leq
\sum_{i=1}^n
\Hh^{n-1}\bigl(\Delta_i(\underline A_i)\bigr).
\end{aligned}
\]
\end{proof}

We next recover the bound of Pelekis and Vlas\'ak~\cite{PelekisVlasak}.

\begin{corollary}\label{cor:cube}
Let $k$ be a positive integer.
If $A\subseteq[0,1]^n$ is a $k$-antichain with respect to $\preceq$, then
\[
\Hh^{n-1}(A)\leq kn.
\]
\end{corollary}

\begin{proof}
Every chain under $\ll$ is also a chain under $\preceq$.
Thus $A$ is a $k$-antichain with respect to $\ll$, and so is each subset $\underline A_i$.
Applying Corollary~\ref{cor:weighted-projection} with $v=(1,\ldots,1)$ gives
\[
\Hh^{n-1}(\underline A_i)
\leq
k\Hh^{n-1}\bigl(\Delta_i(\underline A_i)\bigr)
\qquad(i\in[n]).
\]
By subadditivity and~\eqref{eq:cube-projection-bound},
\[
\Hh^{n-1}(A)
\leq
\sum_{i=1}^n\Hh^{n-1}(\underline A_i)
\leq
k\sum_{i=1}^n
\Hh^{n-1}\bigl(\Delta_i(\underline A_i)\bigr)
\leq kn.
\]
\end{proof}

Janzer~\cite{JanzerContinuous} proved that the bound in Corollary~\ref{cor:cube} is attained when $k=1$.
Campbell~\cite{Campbell} proved that this bound is attained for every positive integer $k$.
We ask for a description of all sets attaining equality.

\medskip
\noindent\textbf{Question.}
For a positive integer $k$, which $\Hh^{n-1}$-measurable $k$-antichains $A\subseteq[0,1]^n$ with respect to $\preceq$ satisfy
\[
\Hh^{n-1}(A)=kn?
\]
\medskip

\textbf{Declaration of AI usage.} During the development of this work, the authors used ChatGPT as a research-assistance tool for exploring possible proof strategies and for language polishing and editing of the manuscript. The resulting arguments were substantially reconstructed, independently verified, and written by the authors.

\section*{Acknowledgements}
The second author was partially supported by the National Natural Science Foundation of China (Nos.\ 12671388 and 12271222).

\end{document}